\documentclass[11pt]{article}

\usepackage{amssymb,amsmath,amstext,amscd,amsthm,latexsym}

\usepackage[dvipdfmx, bookmarkstype=toc, colorlinks=true, linkcolor=cyan, pdfborder={0 0 0}, bookmarks=true, bookmarksnumbered=true]{hyperref}

\DeclareMathOperator*{\regprod}{\mathchoice%
{\ooalign{\hbox{$\displaystyle\prod$}\crcr\hbox{$\displaystyle\coprod$}}}
{\ooalign{\hbox{$\textstyle\prod$}\crcr\hbox{$\textstyle\coprod$}}}
{\ooalign{\hbox{$\scriptstyle\prod$}\crcr\hbox{$\scriptstyle\coprod$}}}
{\ooalign{\hbox{$\scriptscriptstyle\prod$}\crcr\hbox{$\scriptscriptstyle\coprod$}}}
}

\numberwithin{equation}{section}

\theoremstyle{plain}

\newtheorem{theorem}{Theorem}

\newtheorem{proposition}{Proposition}

\newtheorem{remark}{Remark}

\newtheorem{example}{Examples}

\begin{document}

\title{Transcendency of the determinant of the Riemann operator: on higher $K$--groups}

\author{Nobushige Kurokawa\footnote{Department of Mathematics, Tokyo Institute of Technology} \and Hidekazu Tanaka\footnote{6-15-11-202 Otsuka, Bunkyo-ku, Tokyo}}

\date{October 10, 2022}



\maketitle

\begin{abstract}
In previous papers we investigated basic properties of the determinant $G_{K}(s)$ of the Riemann operator: ${\mathcal R}$ acting on $\bigoplus_{n>1} K_{n}(A)_{\mathbb{C}}$, where $A$ is the integer ring of an algebraic number field $K$. The function $G_{K}(s)$ is defined as the regularized determinant
\[
G_{K}(s) = {\rm det} ((s I-\mathcal{R}) | \bigoplus_{n>1} K_{n}(A)_{\mathbb{C}} )
\]
with $\mathcal{R} | K_{n}(A)_{\mathbb{C}} = \frac{1-n}{2}$.

We showed that $G_{K}(s)^{-1}$ is essentially the so called gamma factors of Dedekind zeta function of $K$.

In this paper we study the transcendency of $G_{K}(s)$ for some rational numbers $s$. The result depends on types of $K$. For example, we show that $G_{K}(\frac{1}{3})$ is a transcendental number if $K$ is a totally imaginary and $G_{K}(\frac{1}{2})$ is a transcendental number otherwise.

\end{abstract}

\section*{Introduction} Let $A$ be the integer ring of an algebraic number field $K$. We denote by $K_{n}(A)$ the higher $K$--group of $A$ constructed by Quillen. We define the Riemann operator ${\mathcal R}$ on $K_{n}(A)$ as $\mathcal{R} | K_{n}(A) = \frac{1-n}{2}$. 

In \cite{KT1} \cite{KT2} we studied the regularized determinant
\begin{align*}
G_{K}(s) &= {\rm det} ((s I-\mathcal{R}) | \bigoplus_{n>1} K_{n}(A)_{\mathbb{C}} )\\
&= \regprod_{n>1} (s-\frac{1-n}{2})^{{\rm rank} K_{n}(A)}\\
&= \regprod_{n>1} (\frac{n-1}{2}+s)^{{\rm rank} K_{n}(A)}\\
&= \exp( -\frac{\partial}{\partial w} \varphi(w,s)\biggl|_{w=0} ),
\end{align*}
where 
\[
\varphi(w,s)=\sum_{n>1} {\rm rank}K_{n}(A)(\frac{n-1}{2}+s)^{-w}.
\]
We know from Borel \cite{B} that
\[
{\rm rank} K_{n}(A) = \left\{
\begin{array}{ccc}
r_{1}+r_{2} & {\rm if} & n \equiv 1 \; {\rm mod} \; 4,\\
r_{2} & {\rm if} & n \equiv 3 \; {\rm mod} \; 4,\\
0 & & {\rm otherwise},
\end{array}
\right.
\]
where $r_{1}$ (resp. $r_{2}$) is the number of real (resp. complex) places of $K$. Hence,
\[
\varphi(w,s)=(r_{1}+r_{2})\sum_{n>1 \atop n \equiv 1 \; {\rm mod} \; 4} (\frac{n-1}{2}+s)^{-w} + r_{2} \sum_{n \equiv 3 \; {\rm mod} \; 4} (\frac{n-1}{2}+s)^{-w}.
\]
Calculations of $G_{K}(s)$ were done in \cite{KT1} \cite{KT2} as
\[
G_{K}(s)^{-1} = \biggl( 2^{\frac{s}{2}} \frac{\Gamma(\frac{s+2}{2})}{\sqrt{\pi}} \biggl)^{r_{1}+r_{2}} \biggl( 2^{\frac{s}{2}} \frac{\Gamma(\frac{s+1}{2})}{\sqrt{2\pi}} \biggl)^{r_{2}}.
\]
To express $G_{K}(s)$ more neatly we use notations 
\[
\zeta_{{\mathbb R}}(s) = \biggl( \regprod_{n=0}^{\infty}(2n+s) \biggl)^{-1}=\frac{\Gamma(\frac{s}{2})}{\sqrt{2 \pi}} 2^{\frac{s-1}{2}}
\]
and
\[
\zeta_{{\mathbb C}}(s) = \biggl( \regprod_{n=0}^{\infty}(n+s) \biggl)^{-1}=\frac{\Gamma(s)}{\sqrt{2 \pi}}.
\]
We refer to Manin \cite{M} for regularized products. We prove the following relation:

\begin{theorem} 
\[
\zeta_{{\mathbb C}}(s) = \zeta_{{\mathbb R}}(s) \zeta_{{\mathbb R}}(s+1).
\]
This is equivalent to
\[
\Gamma_{{\mathbb C}}(s) = \Gamma_{{\mathbb R}}(s) \Gamma_{{\mathbb R}}(s+1),
\]
where
\[
\Gamma_{{\mathbb R}}(s) = \Gamma(\frac{s}{2}) \pi^{-\frac{s}{2}}
\]
and
\[
\Gamma_{{\mathbb C}}(s) = \Gamma(s) 2 (2\pi)^{-s}
\]
are usual notations for gamma factors.
\end{theorem}

Then, we can express $G_{K}(s)^{-1}$ as follows.

\begin{theorem}
\[
G_{K}(s)^{-1} = \zeta_{{\mathbb R}}(s+2)^{r_{1}} \zeta_{{\mathbb C}}(s+1)^{r_{2}}.
\]
\end{theorem}

Next, we study the transcendency. First we look at $G_{K}(s)$ for $s \in {\mathbb Z}_{\geq 0}$. We obtain the following result. 

\begin{theorem} Let $n \geq 0$ be an integer.
\par (1) Let $n$ be even. Then $G_{K}(n) \notin \overline{{\mathbb Q}}$.
\par (2) Let $n$ be odd. Then $G_{K}(n) \notin \overline{{\mathbb Q}} \Leftrightarrow r_{2} \geq 1$. In other words, $G_{K}(n) \in \overline{{\mathbb Q}} \Leftrightarrow r_{2}=0 \; (K:{\rm totally} \; {\rm real})$.
\end{theorem}

\begin{example}$\;$
\par (a) $G_{K}(0)={\rm det}(-{\mathcal R}) \notin \overline{{\mathbb Q}}$.
\par (b) $G_{K}(1)={\rm det}(I-{\mathcal R}) \notin \overline{{\mathbb Q}} \Leftrightarrow r_{2} \geq 1$.
\par (c) $G_{K}(0)G_{K}(1)={\rm det}(-{\mathcal R}) {\rm det}(I-{\mathcal R})=(\sqrt{2 \pi})^{[K:{\mathbb Q}]} \notin \overline{{\mathbb Q}}$, where $\sqrt{2\pi}=\regprod_{n=1}^{\infty} n = \infty !$.
\end{example}

Lastly we look at some simple $s \in {\mathbb Q}_{>0}$.

\begin{theorem} Let $n \geq 0$ be an integer.
\par (1) For non--totally imaginary $K$ (i.e., $r_{1} \geq 1$), $G_{K}(\frac{1}{2}+n) \notin \overline{{\mathbb Q}}$, and for totally imaginary $K$, $G_{K}(\frac{1}{2}+n) \in \overline{{\mathbb Q}}$.
\par (2) For totally imaginary $K$, $G_{K}(\frac{1}{3}+n) \notin \overline{{\mathbb Q}}$.
\end{theorem}

Here we need algebraic independency results of Chudnovsky \cite{C}.

\section{Absolute zeta function: Proof of Theorem 1} We use the framework of \cite{KT3} \cite{KT4} concerning the absolute zeta functions (zeta functions over ${\mathbb F}_{1}$) and absolute automorphic forms. We refer to Soul\'{e} \cite{S} and Connes--Consani \cite{CC} for absolute zeta functions in general.

Let $f(x)$ be an absolute automorphic form
\[
f:{\mathbb R} \to {\mathbb C} \cup \{\infty\}
\]
satisfying
\[
f(\frac{1}{x})=C x^{-D} f(x)
\]
with $D \in {\mathbb Z}$ and $C=\pm 1$. We define the absolute Hurwitz zeta function
\[
 Z_{f}(w,s) = \frac{1}{\Gamma(w)}\int_{1}^{\infty} f(x)x^{-s-1}(\log x)^{w-1} dx
\]
and the absolute zeta function
\[
\zeta_{f}(s)=\zeta_{f/{\mathbb F}_{1}}(s)=\exp\Biggl( \frac{\partial}{\partial w} Z_{f}(w,s)\Biggl|_{w=0} \Biggl).
\]

\begin{proposition} Let
\[
f_{{\mathbb C}}(x) = \frac{1}{1-x^{-1}}
\]
and
\[
\zeta_{{\mathbb C}}(s) = \zeta_{f_{\mathbb C}/{\mathbb F}_{1}}(s).
\]
Then
\[
Z_{f_{\mathbb C}}(w,s) = \sum_{n=0}^{\infty}(n+s)^{-w}
\]
and
\[
\zeta_{{\mathbb C}}(s)=\biggl( \regprod_{n=0}^{\infty}(n+s) \biggl)^{-1} = \frac{\Gamma(s)}{\sqrt{2 \pi}}.
\]
\end{proposition}

\begin{proof} Since
\[
f_{{\mathbb C}}(x) = \sum_{n=0}^{\infty} x^{-n}
\]
for $x>1$ we know that
\begin{align*}
Z_{f_{\mathbb C}}(w,s) &= \frac{1}{\Gamma(w)} \int_{1}^{\infty} f_{{\mathbb C}}(x) x^{-s-1} (\log x)^{w-1} dx\\
&= \sum_{n=0}^{\infty} \frac{1}{\Gamma(w)} \int_{1}^{\infty} x^{-n-s-1} (\log x)^{w-1} dx\\
&= \sum_{n=0}^{\infty} (n+s)^{-w}.
\end{align*}
Hence, Lerch's formula gives
\[
\zeta_{{\mathbb C}}(s)=\biggl( \regprod_{n=0}^{\infty}(n+s) \biggl)^{-1} = \frac{\Gamma(s)}{\sqrt{2 \pi}}.
\]
\end{proof}

\begin{proposition} Let
\[
f_{{\mathbb R}}(x) = \frac{1}{1-x^{-2}}
\]
and
\[
\zeta_{{\mathbb R}}(s) = \zeta_{f_{{\mathbb R}}/{\mathbb F}_{1}}(s).
\]
Then
\[
Z_{f_{{\mathbb R}}}(w,s) = \sum_{n=0}^{\infty} (2n+s)^{-w}
\]
and
\[
\zeta_{{\mathbb R}}(s) = \biggl( \regprod_{n=0}^{\infty}(2n+s) \biggl)^{-1} = \frac{\Gamma(\frac{s}{2})}{\sqrt{2 \pi}} 2^{\frac{s-1}{2}}.
\]
\end{proposition}

\begin{proof} Since
\[
f_{{\mathbb R}}(x) = \sum_{n=0}^{\infty} x^{-2n}
\]
for $x>1$ we know that
\begin{align*}
Z_{f_{\mathbb R}}(w,s) &= \frac{1}{\Gamma(w)} \int_{1}^{\infty} f_{{\mathbb R}}(x) x^{-s-1} (\log x)^{w-1} dx\\
&= \sum_{n=0}^{\infty} (2n+s)^{-w}.
\end{align*}
Hence
\[
\zeta_{{\mathbb R}}(s)=\biggl( \regprod_{n=0}^{\infty}(2n+s) \biggl)^{-1}.
\]
Moreover,
\[
Z_{f_{\mathbb R}}(w,s) = 2^{-w} \sum_{n=0}^{\infty} (n+\frac{s}{2})^{-w}
\]
easily gives
\[
\zeta_{{\mathbb R}}(s) = \frac{\Gamma(\frac{s}{2})}{\sqrt{2 \pi}} 2^{\frac{s-1}{2}}.
\]
\end{proof}

\par \underline{Proof of Theorem 1} We present two proofs for the relation $\zeta_{{\mathbb C}}(s)=\zeta_{{\mathbb R}}(s) \zeta_{{\mathbb R}}(s+1)$.
\begin{proof}[First proof] It is sufficient to show
\[
\biggl(\regprod_{n=0}^{\infty}(2n+s) \biggl) \biggl(\regprod_{n=0}^{\infty}(2n+1+s) \biggl)=\regprod_{n=0}^{\infty}(n+s)
\]
since the left hand side is $\zeta_{{\mathbb R}}(s)^{-1} \zeta_{{\mathbb R}}(s+1)^{-1}$ and the right hand side is $\zeta_{{\mathbb C}}(s)^{-1}$. By the way, the left hand side is
\[
\biggl(\regprod_{m \geq 0 \atop {\rm even}}(m+s) \biggl) \biggl(\regprod_{m \geq 0 \atop {\rm odd}} (m+s) \biggl)=\regprod_{m=0}^{\infty}(m+s),
\]
which is nothing but the right hand side. 
\end{proof}

\begin{proof}[Second proof] Remark that
\[
f_{{\mathbb R}}(x)(1+x^{-1})=f_{{\mathbb C}}(x).
\]
Hence we have
\begin{align*}
Z_{f_{\mathbb R}}(w,s) + Z_{f_{\mathbb R}}(w,s+1) &= \frac{1}{\Gamma(w)} \int_{1}^{\infty} f_{{\mathbb R}}(x) x^{-s-1} (\log x)^{w-1} dx\\ 
&+ \frac{1}{\Gamma(w)} \int_{1}^{\infty} f_{{\mathbb R}}(x) x^{-1} x^{-s-1} (\log x)^{w-1} dx\\
&=\frac{1}{\Gamma(w)} \int_{1}^{\infty} f_{{\mathbb R}}(x) (1+x^{-1}) x^{-s-1} (\log x)^{w-1} dx\\
&=\frac{1}{\Gamma(w)} \int_{1}^{\infty} f_{{\mathbb C}}(x) x^{-s-1} (\log x)^{w-1} dx\\
&=Z_{f_{\mathbb C}}(w,s).
\end{align*}
Thus
\begin{align*}
\zeta_{{\mathbb R}}(s) \zeta_{{\mathbb R}}(s+1) &= \exp\biggl( \frac{\partial}{\partial w} \biggl( Z_{f_{\mathbb R}}(w,s) + Z_{f_{\mathbb R}}(w,s+1) \biggl)\biggl|_{w=0} \biggl)\\
&= \exp\biggl( \frac{\partial}{\partial w} Z_{f_{\mathbb C}}(w,s) \biggl|_{w=0} \biggl)\\
&=\zeta_{{\mathbb C}}(s).
\end{align*}
\end{proof}

\begin{proof}[Proof of the equivalence $\zeta_{{\mathbb R}}(s) \zeta_{{\mathbb R}}(s+1)=\zeta_{{\mathbb C}}(s) \Leftrightarrow \Gamma_{{\mathbb R}}(s) \Gamma_{{\mathbb R}}(s+1)=\Gamma_{{\mathbb C}}(s)$]
Notice that 
\[
\Gamma_{{\mathbb R}}(s) = \zeta_{{\mathbb R}}(s) 2^{1-\frac{s}{2}} \pi^{\frac{1-s}{2}}
\]
and
\[
\Gamma_{{\mathbb C}}(s) = \zeta_{{\mathbb C}}(s) 2^{\frac{3}{2}-s} \pi^{\frac{1}{2}-s}.
\]
Hence
\[
\Gamma_{{\mathbb R}}(s) \Gamma_{{\mathbb R}}(s+1)=\zeta_{{\mathbb R}}(s) \zeta_{{\mathbb R}}(s+1) 2^{\frac{3}{2}-s} \pi^{\frac{1}{2}-s}.
\]
Thus we get
\[
\frac{\Gamma_{{\mathbb R}}(s) \Gamma_{{\mathbb R}}(s+1)}{\Gamma_{{\mathbb C}}(s) }=\frac{\zeta_{{\mathbb R}}(s) \zeta_{{\mathbb R}}(s+1)}{\zeta_{{\mathbb C}}(s)}.
\]
Hence
\[
\zeta_{{\mathbb C}}(s)=\zeta_{{\mathbb R}}(s) \zeta_{{\mathbb R}}(s+1) \Leftrightarrow \Gamma_{{\mathbb C}}(s)=\Gamma_{{\mathbb R}}(s) \Gamma_{{\mathbb R}}(s+1).
\]
\end{proof}

\begin{remark} This relation equivalent to the duplication formula for the gamma function also.
\end{remark}

\section{Calculation of $G_{K}(s)$: Proof of Theorem 2}
\begin{proof}[Proof of Theorem 2]
We start from the result in \cite{KT2}:
\[
G_{K}(s)^{-1} = \biggl( 2^{\frac{s}{2}} \frac{\Gamma(\frac{s+2}{2})}{\sqrt{\pi}} \biggl)^{r_{1}+r_{2}} \biggl( 2^{\frac{s}{2}} \frac{\Gamma(\frac{s+1}{2})}{\sqrt{2\pi}} \biggl)^{r_{2}}.
\]
Then we see easily that
\begin{align*}
G_{K}(s)^{-1} &= \zeta_{{\mathbb R}}(s+2)^{r_{1}+r_{2}} \zeta_{{\mathbb R}}(s+1)^{r_{2}}\\
&= \zeta_{{\mathbb R}}(s+2)^{r_{1}} (\zeta_{{\mathbb R}}(s+2) \zeta_{{\mathbb R}}(s+1))^{r_{2}}.
\end{align*}
Hence Theorem 1 gives
\[
G_{K}(s)^{-1} = \zeta_{{\mathbb R}}(s+2)^{r_{1}} \zeta_{{\mathbb C}}(s+1)^{r_{2}}.
\]
\end{proof}

\section{Proof of Theorem 3}
\begin{proof}[Proof of Theorem 3]
\par (1) Let $n \geq 0$ be an even integer. Then
\[
G_{K}(n)^{-1} = \zeta_{{\mathbb R}}(n+2)^{r_{1}} \zeta_{{\mathbb C}}(n+1)^{r_{2}},
\]
where
\[
\zeta_{{\mathbb R}}(n+2) = \frac{\Gamma(\frac{n+2}{2})}{\sqrt{2 \pi}} 2^{\frac{n+1}{2}} \in \overline{{\mathbb Q}}^{\times} \cdot \frac{1}{\sqrt{\pi}}
\]
and
\[
\zeta_{{\mathbb C}}(n+1) = \frac{\Gamma(n+1)}{\sqrt{2 \pi}} \in \overline{{\mathbb Q}}^{\times} \cdot \frac{1}{\sqrt{\pi}}.
\]
Hence
\[
G_{K}(n) \in \overline{{\mathbb Q}}^{\times} \cdot \pi^{\frac{r_{1}+r_{2}}{2}}.
\]
Thus
\[
G_{K}(n) \notin \overline{{\mathbb Q}}.
\]
\par (2) Let $n \geq 0$ be an odd integer. Then
\[
G_{K}(n)^{-1} = \zeta_{{\mathbb R}}(n+2)^{r_{1}} \zeta_{{\mathbb C}}(n+1)^{r_{2}},
\]
where
\[
\zeta_{{\mathbb R}}(n+2) = \frac{\Gamma(\frac{n+2}{2})}{\sqrt{2 \pi}} 2^{\frac{n+1}{2}} \in \overline{{\mathbb Q}}^{\times} 
\]
and
\[
\zeta_{{\mathbb C}}(n+1) = \frac{\Gamma(n+1)}{\sqrt{2 \pi}} \in \overline{{\mathbb Q}}^{\times} \cdot \frac{1}{\sqrt{\pi}}.
\]
Hence
\[
G_{K}(n) \in \overline{{\mathbb Q}}^{\times} \cdot \pi^{\frac{r_{2}}{2}}.
\]
Thus
\[
G_{K}(n) \notin \overline{{\mathbb Q}} \Leftrightarrow r_{2} \geq 1.
\]
\end{proof}

\begin{example}$\;$
\par (a) $G_{K}(0)={\rm det}(-{\mathcal R}) = 2^{\frac{r_{2}}{2}} \pi^{\frac{r_{1}+r_{2}}{2}} \notin \overline{{\mathbb Q}}$.
\par (b) $G_{K}(1)={\rm det}(I-{\mathcal R}) = 2^{\frac{r_{1}+r_{2}}{2}} \pi^{\frac{r_{2}}{2}} \notin \overline{{\mathbb Q}} \Leftrightarrow r_{2} \geq 1$.
\par (c) 
\begin{align*}
G_{K}(0)G_{K}(1) &= {\rm det}(-{\mathcal R}) {\rm det}(I-{\mathcal R})\\
&=2^{\frac{r_{1}+2r_{2}}{2}} \pi^{\frac{r_{1+2}r_{2}}{2}}\\
&=(\sqrt{2 \pi})^{r_{1}+2r_{2}}\\
&=(\sqrt{2 \pi})^{[K:{\mathbb Q}]} \notin \overline{{\mathbb Q}}.
\end{align*}
\end{example}

\begin{remark}$\;$
\par (1) $\zeta_{{\mathbb R}}(n) \in \overline{{\mathbb Q}}^{\times}$ for an odd integer $n \geq 1$.
\par (2) $\zeta_{{\mathbb R}}(n) \in \overline{{\mathbb Q}}^{\times} \cdot \frac{1}{\sqrt{\pi}}$ for an even integer $n \geq 2$.
\par (3) $\zeta_{{\mathbb C}}(n) \in \overline{{\mathbb Q}}^{\times} \cdot \frac{1}{\sqrt{\pi}}$ for an integer $n \geq 1$.
\end{remark}

\section{Transcendency for $s=\frac{1}{2}+n$, $\frac{1}{3}+n$: Proof of Theorem 4}
\begin{proof}[Proof of Theorem 4]
We notice that it is sufficient to look at $n=0,1$ since the ``periodicity'' of $G_{K}(s):G_{K}(s+2)=G_{K}(s)(s+1)^{-r_{1}}(s+2)^{-r_{1}-r_{2}}$.
\par (1) $G_{K}(\frac{1}{2}+n)$
\par Look at
\[
G_{K}(\frac{1}{2}+n)^{-1} = \zeta_{{\mathbb R}}(\frac{1}{2}+n+2)^{r_{1}} \zeta_{{\mathbb C}}(\frac{1}{2}+n+1)^{r_{2}},
\]
where
\begin{align*}
&\zeta_{{\mathbb R}}(\frac{1}{2}+n+2) = \frac{\Gamma(\frac{1}{4}+\frac{n}{2}+1)}{\sqrt{2 \pi}} 2^{\frac{2n+3}{4}} \notin \overline{{\mathbb Q}}\\
&\left\{
\begin{array}{cc}
n=0: & \zeta_{{\mathbb R}}(\frac{1}{2}+2)=2^{-\frac{7}{4}} \cdot \frac{\Gamma(\frac{1}{4})}{\sqrt{\pi}},\\
n=1: & \zeta_{{\mathbb R}}(\frac{1}{2}+3)=3 \cdot 2^{-\frac{3}{4}} \cdot \frac{\sqrt{\pi}}{\Gamma(\frac{1}{4})}
\end{array}
\right.
\end{align*}
by the result of Chudnovsky \cite{C} (the algebraic independency of $\pi$ and $\Gamma(\frac{1}{4})$), and
\[
\zeta_{{\mathbb C}}(\frac{1}{2}+n+1)=\frac{\Gamma(\frac{1}{2}+n+1)}{\sqrt{2 \pi}} \in \overline{{\mathbb Q}}^{\times}.
\]
Hence, if $r_{1} \geq 1$ (i.e., non--totally imaginary $K$), we know that $G_{K}(\frac{1}{2}+n) \notin \overline{{\mathbb Q}}$. Moreover if $r_{1}=0$ (i.e., totally imaginary $K$), $G_{K}(\frac{1}{2}+n) \in \overline{{\mathbb Q}}$.
\par (2) $G_{K}(\frac{1}{3}+n)$
\par Look at
\[
G_{K}(\frac{1}{3}+n)^{-1} = \zeta_{{\mathbb R}}(\frac{1}{3}+n+2)^{r_{1}} \zeta_{{\mathbb C}}(\frac{1}{3}+n+1)^{r_{2}},
\]
where
\[
\zeta_{{\mathbb R}}(\frac{1}{3}+n+2) = \frac{\Gamma(\frac{1}{6}+\frac{n+2}{2})}{\sqrt{2 \pi}} 2^{\frac{2}{3}+\frac{n}{2}}
\]
and
\[
\zeta_{{\mathbb C}}(\frac{1}{3}+n+1) = \frac{\Gamma(\frac{1}{3}+n+1)}{\sqrt{2 \pi}}.
\]
We see that 
\[
\zeta_{{\mathbb C}}(\frac{1}{3}+n+1) \notin \overline{{\mathbb Q}}
\]
by the result of Chudnovsky \cite{C} (the algebraic independency of $\pi$ and $\Gamma(\frac{1}{3})$). [We do not know the transcendental nature of $\zeta_{{\mathbb R}}(\frac{1}{3}+n+2)$.] Hence, if $r_{1} = 0$ (i.e., totally imaginary $K$), then we know that $G_{K}(\frac{1}{3}+n) \notin \overline{{\mathbb Q}}$.
\end{proof}


\begin{thebibliography}{99}

\bibitem[B]{B} A. Borel ``Stable real cohomology of arithmetic groups'' Ann. Sci. \'{E}cole Norm. Sup. (4) {\bf 7} (1974) 235--272.
 


\bibitem[C]{C} G. V. Chudnovsky ``Contributions to the Theory of Transcendental Numbers'' Math. Surv. Monographs {\bf 19}, 1984, Amer. Math. Soc., Providence.

\bibitem[CC]{CC} A. Connes and C. Consani ``Schemes over ${\mathbb F}_{1}$ and zeta functions'' Compositio Math. {\bf 146} (2010) 1383--1415.

\bibitem[KT1]{KT1} N. Kurokawa and H. Tanaka ``Riemann operators on higher $K$--groups'' [arXiv:2209.12837]

\bibitem[KT2]{KT2} N. Kurokawa and H. Tanaka ``Determinants of Riemann operators on Quillen's higher $K$--groups: periodicity'' [arXiv:2209.13843]

\bibitem[KT3]{KT3} N. Kurokawa and H. Tanaka ``Absolute zeta functions and the automorphy'' Kodai Math. J. 40 (2017), no. 3, 584-614. 

\bibitem[KT4]{KT4} N. Kurokawa and H. Tanaka ``Absolute zeta functions and absolute automorphic forms'' J. Geom. Phys. 126 (2018), 168-180. [Connes 70]

\bibitem[M]{M} Y. Manin ``Lectures on zeta functions and motives (according to Deninger and Kurokawa)'' Asterisque, {\bf 228} (1995), 121--163.

\bibitem[S]{S} C. Soul\'{e} ``Les vari\'{e}t\'{e}s sur le corps \`{a} un \'{e}l\'{e}ment'' Mosc. Math. J. {\bf 4} (2004) 217--244.
\end{thebibliography}
\end{document}